\documentclass[11pt]{amsart}
\usepackage{graphicx, amssymb,ucs}
\numberwithin{equation}{section}
\usepackage{float}
 \usepackage[utf8x]{inputenc}
 \usepackage{graphics}
\font\tenms=msbm10
\font\sevenms=msbm7
\font\fivems=msbm5
\newfam\bbfam \textfont\bbfam=\tenms \scriptfont\bbfam=\sevenms
\scriptscriptfont\bbfam=\fivems

\newtheorem{Thm}{Theorem}
\newtheorem{Def}{Definition}[section]
\newtheorem{Rem}[Def]{Remark}
\newtheorem{Prop}[Def]{Proposition}

\newtheorem{Lem}[Def]{Lemma}

\def\d {{\partial}}

\newcommand{\ba}{\begin{aligned}}
\newcommand{\ea}{\end{aligned}}
\newcommand{\be}{\begin{equation}}
\newcommand{\ee}{\end{equation}}

\def \N{{\mathbf N}}
\def \R{{\mathbf R}}

\def\cK{\mathcal{K}}

\def \supp{\hbox{{\rm Supp}}}

\def \eps{{\varepsilon}}
\def \e{{\varepsilon}}

\def \CC{{\mathcal C}}

\def \cK{{\mathcal K}}

\def \d{{\partial}}

\newcommand{\xib}{\langle \xi\rangle_b}

\newcommand{{\bv}}{ \bf v}

\begin{document}

\title[On the propagation of oceanic waves]{On the propagation of oceanic waves driven by a strong macroscopic flow}

\author[I. Gallagher]{Isabelle Gallagher}
\address[I. Gallagher]%
{Institut de Math{\'e}matiques UMR 7586 \\
      Universit{\'e} Paris VII \\
175, rue du Chevaleret\\
75013 Paris\\FRANCE}
\email{Isabelle.Gallagher@math.jussieu.fr}
\author[T. Paul]{Thierry Paul}
\address[T. Paul]{CNRS and Centre de mathématiques Laurent Schwartz \\
 \'Ecole polytechnique, 91128 Palaiseau cedex \\FRANCE}
\email{thierry.paul@math.polytechnique.fr}
\author[L. Saint-Raymond]{Laure Saint-Raymond}
\address[L. Saint-Raymond]%
{Universit\'e Paris VI and DMA \'Ecole Normale Sup\'erieure, 45 rue d'Ulm, 75230
Paris Cedex 05\\FRANCE }
\email{Laure.Saint-Raymond@ens.fr}

 \begin{abstract}
In this work we study oceanic waves in a shallow water  flow subject to strong wind forcing and rotation, and linearized around a inhomogeneous (non zonal) stationary profile. 
 This   extends the study~\cite{CGPS}, where the profile was assumed to be zonal only and where explicit calculations were made possible due to the 1D setting.
 
 Here the diagonalization of the system, which allows to identify Rossby and Poincar\'e waves, is  proved by an abstract semi-classical approach. The dispersion of Poincar\'e waves is also obtained   by a more abstract and more robust method using Mourre estimates. Only some partial results however are   obtained concerning the Rossby propagation, as the two dimensional setting  complicates very much the study of the dynamical system.
  \end{abstract}

\maketitle

\section{Introduction}
This paper is a continuation of~\cite{CGPS} so before discussing the matter of this paper (in Section~\ref{results}) let us   review the contents of  that work.  We shall start by recalling briefly the model, then we shall explain the methods and results obtained in~\cite{CGPS} and discuss their limitations. 

\subsection{The model}
The goal of~\cite{CGPS}  is   to understand, through the study of a toy model, the persistence of oceanic eddies  observed long past by physicists among which~\cite{gill,gill-longuet,Greenspan,majda,P1,P2}, who gave heuristic arguments  to explain their formation due both to   wind forcing and to   convection by a macroscopic current.

The ocean is considered in this toy model as an incompressible, inviscid fluid with free surface submitted to gravitation and wind forcing,
and we further make the following classical assumptions: we assume that  the density of the fluid is homogeneous  $\rho =\rho_0 =\hbox{constant}$, that the pressure law is given by the hydrostatic
approximation $p=\rho_0  g z $, and that the motion is essentially horizontal and does not
depend on the vertical coordinate. This leads   to the so-called   shallow water approximation.
 
For the sake of simplicity,  the effects of the interaction with the boundaries are not discussed and the model is purely horizontal with the longitude~$x_1$ and the latitude~$x_2$ both in~$\R$.
\bigskip

The evolution of the water height $h$ and velocity $v$ is then governed by
the  shallow-water equations with Coriolis force 
\begin{equation}
\label{SW}
\begin{aligned}
\d_t( \rho_0 h) +\nabla\cdot (\rho_0hv) =0\\
\d_t (\rho_0hv) +\nabla\cdot  (\rho_0 hv \otimes v) + \omega(\rho_0 hv)^\perp + \rho_0 g h\nabla h =\rho_0h\tau
\end{aligned}
\end{equation}
where $\omega $ denotes the vertical component of the Earth rotation vector $\Omega$, $v^\perp := (-v_2,v_1)$, $g$ is the gravity
 and
$\tau$ is the - stationary - forcing responsible for  the macroscopic flow.  The vertical component of the Earth rotation is therefore~$\Omega \sin (x_2/R)$, where~$R$ is the radius of the Earth; note that  it is classical in the physical literature
to consider the linearization of~$\omega $ (known as the betaplane approximation)~$\omega(x_2) =\Omega  x_2/R$. We consider general functions~$\omega$ in the sequel, with some restrictions that are be made precise later.

We    consider  small fluctuations  around the stationary solution~$(\bar h, \bar v)$ satisfying
$$ \bar h  = \mbox{constant},\quad \nabla\cdot  (\bar v\otimes \bar v) +  {\omega} \bar v^\perp  =\tau  , \quad \mbox{div} \: \bar v = 0.$$
   In~\cite{CGPS} the study is restricted  to the case of a shear flow, in the sense that 
$
\bar v (x)= (\bar v_1(x_2),0),
$
with~$\bar v_1$   a smooth, compactly supported function.  Moreover  some orders of magnitude and scalings   allow to transform the previous system into the following one:
\begin{equation}
\label{SV2}
\begin{aligned}
  \d_t \eta +  \frac1{\eps}  \nabla \cdot u +  \bar u \cdot \nabla \eta +\eps^2 \nabla \cdot (\eta u)=0\,,\\
 \d_t u  +      \frac1{\eps^2}  b  u^\perp  +\frac1\eps  \nabla \eta +\bar u \cdot \nabla u + u\cdot \nabla \bar u +\eps^2 u \cdot \nabla u =0\,,
\end{aligned}
\end{equation}
where~$b:=\omega/|\Omega|$ and where~$\e$ is a small parameter (of the order of~$\mbox{Fr}^2$, where $\mbox{Fr}$ is the Froude number, and of~$\mbox{Ro}^\frac12$  where~$\mbox{Ro} $ is the Rossby number).

\subsection{Methods and results in~\cite{CGPS}}

Most of the analysis  in~\cite{CGPS}  concerns the linear version of~(\ref{SV2}), namely the following system:
     \begin{equation}\label{linsystcgps}
\displaystyle \eps^2 i\partial_t{{\bv}} + A(x_2,\eps D, \eps){\bv} =  0 \,  \qquad 
{\bv} = (v_0,v_1,v_2) ,
\end{equation} 
where~$D := \frac 1 i \partial$, and the linear propagator is given by
$$
A (x_2,\eps D, \eps):=   i\left( \begin{matrix}   \eps\bar u_1  \eps \partial_1  &\eps \d_1&\eps \d_2 \\
\eps \d_1 & \eps\bar u_1  \eps \partial_1   & -b( x_2)+\eps^2   \bar u_1' \\
\eps \d_2 & b( x_2)  &  \eps\bar u_1  \eps \partial_1 \end{matrix} \right) \,,
$$

The first step of the analysis  consists in diagonalizing (approximately) the system~(\ref{linsystcgps}). The computation of a kind of {\bf characteristic polynomial} associated with~(\ref{linsystcgps}), in symbolic form, allows to construct three  symbols the quantization of which provides three scalar propagators (this will be explained more explicitly below). 

Two of those propagators, called Poincar\'e propagators, are then proved to satisfy dispersive estimates; that result relies on a spectral analysis (usual semi-classical theory does not operate here due to the very large time scales at play) using {\bf Bohr-Sommerfeld quantization}, which requires that~$b$ has only one, non degenerate critical value and which also uses very much the fact that the motion is translation-invariant in~$x_1$. A stationary phase argument on the spectral decomposition of any solution to the Poincar\'e propagation gives the result: Poincar\'e modes exit any compact set in finite time. 

The last propagator is the Rossby one, which is one order of magnitude (in~$\e$) smaller than the Poincar\'e modes. This allows to analyse the propagation by    semi-classical analysis tools. In particular the precise study of the dynamical system associated with those waves, which is an {\bf integrable system} due to translation invariance in~$x_1$,  allows to derive a   condition on the initial microlocalization of the solution which guarantees that the Rossby waves are trapped for all times in a compact set.

Those results on the linear system~(\ref{linsystcgps}) can finally be transposed   to the original system~(\ref{SV2}) due to the high power of~$\e$ in front of the nonlinearity, and due to the semi-classical setting, which allows to exhibit  vector fields which almost-commute with the linear operator~$A (x_2,\eps D, \eps)$.

\subsection{Limitations of the methods of~\cite{CGPS} }
The restriction which is the most used in the   analysis described briefly in the previous paragraph is the fact that the stationary flow~$\bar u$ is a shear flow of the type~$\bar u = (\bar u_1( x_2) , 0)$. Indeed  

\begin{itemize}

\item It allows to Fourier-transform in the direction~$x_1$, which makes the diagonalization procedure   much easier;

\item  It   simplifies the spectral analysis of Poincar\'e waves, again due to the Fourier transform (in particular the dual variable~$\xi_1 $ is fixed during the propagation, and there is a wave-like behaviour in~$x_1$);

\item  It   allows   the Rossby dynamical system to be integrable, which is a tremendous help in the analysis. 

\end{itemize}

An additional restriction in the previous arguments is that in order to prove the dispersion of  Poincar\'e waves, the rotation amplitude~$b$ should have only one, non degenerate critical value: this allows to use a Bohr-Sommerfeld quantization argument to compute the eigenvalues of the Poincar\'e operator. This assumption on $b$  is not really restrictive from the physical point of view. On the other hand, it is important for physical reasons to consider 2D convection flows.

\subsection{On the nonlinear term}
As explained above, most of the analysis in~\cite{CGPS} is concerned with the  linear system~(\ref{linsystcgps}). In order to transpose the linear results to the nonlinear setting, one uses the following arguments (along with the fact that the coupling is vanishes when~$\e$ goes to zero):

\begin{itemize}

\item Uniform existence which is obtained via an almost-commutation result;

\item  Bilinear estimates in anisotropic semi-classical spaces;

\item   A Gronwall lemma, which requires an~$L^\infty(\R^2)$ bound on the linear solution. This is not known in general, due to the bad Sobolev embeddings in   semi-classical settings, so the nonlinear result is proved for  vanishing couplings only.
\end{itemize}

It is important to notice that none of those three steps require  that~$\bar u $ is a shear flow. In the whole of this paper we shall therefore only focus on the linear equation, and leave to the reader the transposition to the nonlinear equation, using the above steps.

 \section{Main result of this  paper and strategy of the proof}\label{results}
\subsection{The model}
   In this paper we shall be concerned with the linear system
     \begin{equation}\label{linsyst}
\displaystyle \eps^2 i  \partial_t{\bv} + A(x,\eps D, \eps){\bv} =  0 \,  \qquad 
{\bv} = (v_0,v_1,v_2) ,
\end{equation} 
where the linear propagator is given by
\begin{equation}\label{defA}
A (x,\eps D, \eps):=  i  \left( \begin{matrix}   \eps\bar u\cdot \eps \nabla  &\eps \d_1&\eps \d_2 \\
\eps \d_1 & \eps\bar u\cdot \eps \nabla +\eps^2 \partial_1 \bar u_1  & -b( x_2)+\eps^2 \partial_2 \bar u_1 \\
\eps \d_2 & b( x_2) +\eps^2 \partial_1 \bar u_2&  \eps\bar u\cdot \eps \nabla +\eps^2 \partial_2 \bar u_2 \end{matrix} \right) \,.
\end{equation}

We shall assume throughout the paper that~$b$ is smooth, with a  symbol-like behaviour: for all~$ \alpha  \in \N $, there is a constant~$ C_\alpha$ such that for all~$ x_2\in \R $,
\begin{equation}\label{bsymbol}
    |b^{(\alpha )} (x_2)|  \leq C_\alpha \bigl(1+b^2(x_2) \bigr)^\frac12.
 \end{equation}
We shall further assume that
$$
\lim_{|x_2| \to \infty} b^2(x_2) = \infty,
$$ 
and that $b^2$ has only non degenerate critical points.

We shall also  suppose that the initial data is microlocalized in  some compact set~$\mathcal C$ of~$T^* \R^{2}   $   satisfying
\begin{equation}\label{cond1}
\mathcal C\cap\{\xi_1^2 + \xi_2^2 + b^2(x_2)=0\} = \emptyset \,
\end{equation}
or actually rather
\begin{equation}\label{cond2}
\mathcal C\cap\{\xi_1=0\} = \emptyset \,.
\end{equation}
We shall prove that  assumption  (\ref{cond1}) is propagated by the  flow, while (\ref{cond2}) is propagated only by the Poincar\'e component.
    We recall (see for instance~\cite{CGPS},  Appendix B) that a function~$f$ is  microlocalized  in a  compact set~${\mathcal C}$ of~$T^* \R^{2}   $ if
for any~$(x_0,\xi_0)  $ in the complement of~${\mathcal C}$ in~$  \R^{4}   $ (we shall identify ~$T^* \R^{2}   $ to~$\R^4$ in the following),
  there is a smooth function~$\chi_0 $, bounded as well as all its derivatives and equal to one at~$(x_0,\xi_0)$, satisfying
\begin{equation}\label{microloc}
\| {\rm Op}^W_\eps (\chi_0)u_0 \|_{L^2(\R^{2})} = O(\eps^\infty),
\end{equation}
where~${\rm Op}^W_\eps$ denotes the Weyl quantization:
\begin{equation}\label{weyl}
 {\rm Op}^W_\eps(\chi_0) u_0(x ) := \frac1{(2\pi \eps)^4} \int e^{i(x-y) \cdot \xi/\eps} \chi_0 (\frac{x+y}2,\xi) u_0(y) \: dyd\xi
.
\end{equation}
We also recall that~(\ref{microloc}) means that for any~$N \in \N$, there are~$\e_0$ and~$C$ such that
$$
  \forall \eps \in  ]0,\eps_0], \quad \| {\rm Op}^W_\eps (\chi_0)u_0 \|_{L^2(\R^{2})}  \leq C \e^N .
$$
    In the following, to simplify some formulations, we shall denote by~$(\mu)\mbox{Supp}_{\star} f$ the projection of the (micro)support of~$f$ onto the~$\star = 0$ axis.

    \subsection{Statement of the main result and organization of the paper}
    
        Let us state the   main theorem proved in   this paper.
    \begin{Thm} \label{mainresult} 
Let~${\bv}_{\e,0} $ be a family of initial data, microlocalized in a compact set~${\mathcal C}$  satisfying Assumptions~(\ref{cond1})-(\ref{cond2}). For any parameter~$\eps >0$, denote by ${\bv}_\eps$ the associate solution to~(\ref{linsyst}).
Then for all~$t  \geq 0$ one can write~${\bv}_\e(t)$ as the sum of a ``Rossby" vector field and a ``Poincar\'e" vector field: $ {\bv}_\e(t)  = {\bv}_\e^R(t)  + {\bv}_\e^P(t) $, satisfying the following properties:
 
\begin{enumerate}
\item \label{nondegeneracy}
$\mu \mbox{Supp}\:    {\bv}_\e^R(t) $ and  $\mu \mbox{Supp} \:  {\bv}_\e^P(t) $ satisfy~(\ref{cond1}) for all times.

\item \label{poincarebehaviour}
For any compact set~$\Omega$ in~$\R^2$, one has
$$
\forall  t> 0,\quad  \| {\bv}_\e^P(t)\|_{L^2(\Omega)}=O(\e^\infty).
$$

\item \label{rossbybehaviour}
 $\mu \mbox{Supp}_{x_2}   {\bv}_\e^R(t) $ lies in a   bounded subset of~$\R$ uniformly in time.
\end{enumerate}
\end{Thm}

\medskip

Compared to~\cite{CGPS}, the main difficulties are due to the presence of a~$x_1$-dependent underlying flow~$\bar u$. The diagonalization of the system (exhibiting Rossby and Poincar\'e-type waves, with very different qualitative features) must be revised, and obtained in a less explicit way. Moreover the proof of~(\ref{poincarebehaviour}) in Theorem~\ref{mainresult}, namely the  dispersion  of Poincar\'e waves can also not be proved in the same way (note that it is not assumed here   that~$b^2$ has only one non degenerate critical value). Finally the trapping of Rossby waves seems much harder to obtain since the underlying dynamical system no more decouples; the behaviour of the Rossby waves is therefore much less precise than in~\cite{CGPS}. 

Let us  explain our strategy here, compared with that in~\cite{CGPS} described above.

 \subsubsection{The diagonalization}
The construction
of the Rossby and Poincar\'e  modes   is not  as direct as in~\cite{CGPS}  due to the lack of translation invariance in~$x_1$. We choose therefore to follow a more abstract way to recover those modes in Section~\ref{scalar}, which relies on {\bf semi-classical analysis, and normal forms} (instead of explicit computations as in~\cite{CGPS}).  Finding the propagators associated with those modes requires a microlocalisation assumption of the type~(\ref{cond1}), in order for the eigenvalues of the matrix of principal symbols to be well separated. The diagonalization result is therefore in this paragraph conditional to the fact that the solution to the propagation equation is correctly microlocalized (that corresponds to Point~\ref{nondegeneracy} of Theorem~\ref{mainresult}).

 \subsubsection{Dispersion of Poincar\'e waves and propagation of the nondegeneracy assumption~(\ref{cond2})} 
 In order to prove~(\ref{poincarebehaviour}) in Theorem~\ref{mainresult} we again rely on a more abstract, and more efficient method than that followed in~\cite{CGPS}. It is based on Mourre estimates and the assumption~(\ref{cond2}) on the initial data: we start by proving, by a semi-classical argument, that after a very short time (of the order of~$\eps$)  the support in~$x_1$ of the solution escapes the support of~$\bar u$. Then we use {\bf Mourre estimates} to prove that the solution remains outside the support of~$\bar u$ for all times, and actually escapes any compact set in~$x_1$ in finite time
   (to prove this last point we use the fact that the equation reduces to a  translation-invariant equation in~$x_1$ since the support of the solution is outside the support in~$x_1$ of~$\bar u$).  This allows finally
   to check that  the nondegeneracy assumption~(\ref{cond2}) does hold for all times.
   This analysis is achieved in Section~\ref{sectionpoincare}.

 \subsubsection{Study of Rossby waves and propagation of the nondegeneracy assumption~(\ref{cond1})}
 In  Section~\ref{rossby} we first prove that the  nondegeneracy assumption~(\ref{cond1}) does hold during the propagation of Rossby waves.
 That is due to semi-classical analysis, by the study of the dynamical system associated with those waves. The study of that system is also the key to the proof of Point~(\ref{rossbybehaviour}), which is also proved in Section~\ref{rossby}.

\section{Reduction to scalar propagators}\label{scalar}
   
   In this section we shall construct three operators~$T_+$, $T_-$ and~$T_R$ diagonalizing~$A(x,\e D, \e)$. 
   
   We shall start by proving a general
  diagonalization result, and at the end we shall apply the general result to our context.

   Before stating the general result, let us give some notation. 
A  semi-classical symbol is a function~$a = a(x,\xi;\varepsilon)$ defined on~$\R^{2d} \times ]0,\eps_0]$ for some~$\eps_0 >0$, which depends smoothly on~$(x,\xi)$ and such that for any~$\alpha \in \N^{2d}$ and any compact $\cK \subset \R^{2d}$, there is a constant~$C$ such that for any~$((x,\xi),\eps) \in \cK\times ]0,\eps_0]$, 
$$
 |\partial^\alpha a((x,\xi),\eps)| \leq C .
$$
We shall consider the Weyl quantization of such symbols, as recalled in~(\ref{weyl}): for all $u$ is in~$ \mathcal D(\R^d)$, 
 $$
 {\rm Op}^W_\eps(a) u (x ) := \frac1{(2\pi \eps)^d} \int e^{i(x-y) \cdot \xi/\eps} a (\frac{x+y}2,\xi) u(y) \: dyd\xi.
 $$
 We shall say that a pseudodifferential operator~$ {\rm Op}^W_\eps(a) $ is supported in a   set~${\mathcal K} $ if for any smooth function~$\chi  $ equal to one in a neighborhood of~${\mathcal K} $ one has~$a\chi = a$.
 
   Finally we shall say that a matrix is  pseudodifferential  if each of its entries 
   is a pseudodifferential operator.

    \medskip 
    Let us  first prove the  following general result. 
   
     \begin{Thm}\label{dia}
  Let~${\mathcal K} $ be a compact subset of~$\R^{2d}$, and consider a~$N \times N$  hermitian pseudodifferential matrix~$A_\e = A(x,\eps D,\eps)$, supported in~${\mathcal K}$. Assume that 
  \begin{itemize}
  \item the (matrix) principal symbol of $A(x,\eps D,0)$, denoted by~$\mathcal A_0$, 
is diagonalizable, in the sense that there are some unitary and diagonal matrices of symbols, $\mathcal U$ and~$\mathcal D$, such that
$$
\mathcal U^{-1}\mathcal A_0\mathcal U=\mathcal D,
$$
\item   the eigenvalues~$(\delta_1(x,\xi),\dots,\delta_N(x,\xi))$  satisfy
\be\label{nondeg}
\forall i \neq j, \quad \inf_{(x,\xi) \in {\mathcal K}}{\vert\delta_i(x,\xi)-\delta_j(x,\xi)\vert}\geq C>0.
\ee
\end{itemize}
Then there exists a family of unitary and diagonal pseudodifferential operators  $V_\e $  and     $D_\e $ supported in~${\mathcal K}$, such that:
\be
V_\e ^{*}A_\e V_\e=D_\e +O(\e^\infty),\quad
V_\e ^{*} V_\e=I +O(\e^\infty).
\ee
Moreover one has
\be
D_\e = {\rm Op}^W_\eps(\mathcal D) +\e D_1+O(\e^2),
\ee
where
the principal symbol of $D_1$ is given by
 $$\mathcal D_1=\hbox{ diag} \left(  \widetilde\Delta_1-\frac{D_0 I_1+I_1D_0 }2\right)$$
 with the notations
\be\label{ddelta}
\begin{aligned}
\widetilde\Delta_1=\frac1 \eps \left(  {\rm Op}^W_\eps(\mathcal U^*)A_\e{\rm Op}^W_\eps(\mathcal U)-D_0  \right),\\
I_1=\frac1 \eps \left(  {\rm Op}^W_\eps(\mathcal U^*){\rm Op}^W_\eps(\mathcal U)-I \right)
\end{aligned}
\ee
More explicitly, let us denote by 
 $a_{ij}(x,\xi)$ the matrix elements of
$\mathcal A_1$, subsymbol of $A(x,\e D,\e)$ defined by:
\[
\mathcal A_1(x,\xi):=\frac{\partial\mathcal A}{\partial\e}(x,\xi,0)
\]
and by $u_{nj}(x,\xi),\ i=1\dots N,$ the coordinates of any unit eigenvector of $\mathcal A_0(x,\xi)$ of eigenvalue $\delta_n(x,\xi)$. We have
\be\label{theformula}
(\mathcal D_1)_{nn}=
\sum_{jk=1\dots N}\left(\Im\left(\overline{u_{jn}}
\{a_{jk},u_{kn}\}\right)+\frac{a_{jk}\{\overline{u_{jn}},u_{kn}\}}{2i}\right)
+(\mathcal
U^*\mathcal A_1\mathcal U)_{nn},
\ee
where $\{f,g\}:=\nabla_\xi f\nabla_xg-\nabla_x f\nabla_\xi g$ is the Poisson bracket on $T^*\R^n$.
%
%
%
%
 \end{Thm}
 
Here and in all the sequel, we say that a pseudo-differential operator $V$ is unitary if it satisfies
$$
V^* V = I +  O(\e^\infty).
$$  
  
  \smallskip
The proof is divided into two parts: in Section \ref{formal} we   present the formal construction and in Section~\ref{symbol} we   show that the symbols of the various operators formally constructed are indeed symbols. Finally
Section~\ref{rosscase} is devoted to the case of the matrix given by~(\ref{defA}).

\subsection{The formal construction}\label{formal}
The proof of Theorem \ref{dia} is a combination of semiclassical and perturbation methods. 
Let us start by defining
$$
U_0 = {\rm Op}^W_\eps(\mathcal U).
$$
Elementary properties of the Weyl quantization imply then that
$
\displaystyle U_0^* A_\e U_0 = D_0 + O(\e).
$

\medskip
 The following proposition
shows that one can construct  a unitary pseudodifferential operator ~$U_\infty$ such that
$$
U_ \infty ^* A_\e U_ \infty = D_0 +  O(\e).
$$

\begin{Lem}\label{unit}
Let $U$ be a pseudodifferential matrix such that $U^*U =I+\eps I_1$, where $I$ is the identity. Then one can find $\displaystyle V \sim\sum_{k=0}^\infty\eps^kV_k$ such that
\be
(U+\eps V )^*(U+\eps V )=I+O(\eps^\infty).
\ee
\end{Lem}\begin{proof}
Let us denote $\displaystyle V_0:=-\frac 1 2 I_1U$. On easily check that $(U+\eps V_0)^*(U+\eps V_0)=I+O(\eps^2)$. Indeed
\begin{eqnarray*}
(U+\eps V_0)^*(U+\eps V_0) & =& U^*U+\frac \eps 2(I_1U^*U+U^*UI_1)+O(\eps^2) \\
&=& I+\eps I_1-\eps I_1+O(\eps^2).
\end{eqnarray*}
Then one
concludes by iteration.
\end{proof}

That lemma allows to define the pseudo-differential operator of (semiclassical) order 0
$$
\Delta_1 = \frac1\eps \left(U_ \infty ^* A_\e U_ \infty - D_0 \right),
$$
where~$U_\infty$ is a unitary operator.

\medskip
Now our aim is to find a unitary operator~$V_\infty$ (up to~$O(\eps^\infty)$) such that
$$
(U_\infty V_\infty)^* A_\e (U_\infty V_\infty) = D_\infty + O(\eps^\infty),
$$
where~$D_\infty = D_0 + \e D_1 + \dots$ is a diagonal matrix satisfying the conclusions of the theorem.

We shall  write~$V_\infty = e^{i\e W }$, with~$W $ selfadjoint (so $V_\infty $ thus constructed is automatically unitary). We recall that if~$W $ is a pseudodifferential operator, then so is~$e^{i\e W}$ (simply by writing~$e^{i\eps W}\sim\sum_0^\infty\frac{(i\eps)^k}{k!}W ^k$).

 We look for~$W $ under the form~$\displaystyle
W\sim\sum_0^\infty \eps^kW_k,$
and compute the $W_k$ recursively.
Since
\[
V_\infty^{*}(D_0+\eps\Delta_1)V_\infty=(D_0+\eps\Delta_1)+i\eps[(D_0+\eps\Delta_1),W]+\frac{(i\eps)^2}2[[(D_0+\eps\Delta_1),W],W]+\dots
\]
we see that, if $W_1$ satisfies
\be\label{eq1}
[D_0,W_1]+\Delta_1=D_1+O(\e^2),\ D_1 \mbox{ diagonal,}
\ee
then we have that
\be\label{eq2}
e^{-i\eps W_1}(D_0+\eps\Delta_1)e^{i\eps W_1}=D_0+\eps D_1+\eps^2\Delta_2,
\ee
where $\Delta_2$ is a zero order pseudodifferential operator. The following lemma is a typical normal form type result, and is crucial for the following.
\begin{Lem}\label{pert}
Let $D_0$ be a diagonal pseudodifferential matrix whose principal symbol $\mathcal D_0 $ has a spectrum satisfying (\ref{nondeg})
and let~$\Delta_1$ be a  pseudodifferential matrix.

 Then
there exist two pseudodifferential matrices $W$ and $D_1$, with~$D_1$ diagonal, such that:
\be\label{secular}
[D_0,W]+ \Delta_1= D_1+\eps\Delta_2,
\ee
where~$\Delta_2$ is  a pseudodifferential matrix of order 0.

Moreover the principal symbol of $D_1$ is the diagonal part of the principal symbol  of
$\Delta_1$.
\end{Lem}
\begin{proof}
By the non degeneracy condition of the spectrum of $\mathcal D_0$ we know, by standard arguments (see \cite{RS} for instance), that there exists a matrix
$\mathcal W_0$ and a diagonal one $\mathcal D_1$ such that
\[
[\mathcal D_0,\mathcal W_0]+\mathcal D_{1,0}=\mathcal D_1,
\]
where $\mathcal D_{1,0}$ is the principal symbol of $\Delta_{1}$.

Indeed it is enough to take $\mathcal D_1$ as the diagonal part of $\mathcal D_{1,0}$ and 
\be \label{normal-form}
(\mathcal W_0(x,\xi))_{i,j}=\frac{(\mathcal D_{1,0}(x,\xi))_{i,j}}{\delta_{i}(x,\xi)-\delta_{j}(x,\xi)} 
\ee
and notice that the Weyl quantization of $\mathcal W_0$ satisfies (\ref{secular}).   
\end{proof}

By Lemma \ref{pert} we know that there exists $W_1$ satisfying (\ref{eq1}). Developing (\ref{eq2}) as
\[
e^{-i\eps W_1}(D_0+\eps\Delta_1)e^{i\eps W_1}=D_0+\eps(\Delta_1+ [D_0,W_1])+\eps^2\Delta_2,
\]
we get immediately (\ref{eq2}).

It is easy to get convinced  that all the $W_k$ will satisfy recursively an equation of the form
\[
[D_0,W_k]+\Delta_k=D_k+O(\eps),
\]
which can be solved by Lemma \ref{pert}.

\medskip
The expression for the principal symbol of $D_1$ follows by construction and the following well known lemma (see \cite{MA} for instance):
\begin{Lem}\label{sub}
Let $a$ and $b$ two symbols. Then the principal symbol of ${\rm Op}^W_\eps(a){\rm Op}^W_\eps(b)$ is $ab$ and its subprincipal
symbol is $\frac 1{2i}\{a,b\}$.
\end{Lem}

In order to derive \eqref{theformula} we have to compute the subprincipal symbol of the diagonal part of the right-hand side
 of \eqref{ddelta}, that is, for each $n=1\dots N$, 
\[
\sum_{jk}
{\rm Op}^W_\eps(\overline{\mathcal U_{jn}}){\rm Op}^W_\eps((\mathcal A_0+\e\mathcal A_1)_{jk}){\rm Op}^W_\eps(\mathcal U_{kn}),
\]
since $\mathcal U$ is unitary.

The term $\e\mathcal A_1$ is obviously responsible for the second term in the right-hand side of \eqref{theformula}. Using Lemma
\ref{sub} and the distributivity of the Poisson bracket, we  get the following expression for the first one:
\[
\sum_{jk}\frac 1 {2i}\left(\{\overline{\mathcal U_{jn}},(\mathcal A_0)_{jk}\mathcal U_{kn}\}
+\overline{\mathcal U_{jn}}\{(\mathcal A_0)_{jk},\mathcal U_{kn}\}\right)
\]
\[
=\sum_{jk}\frac 1 {2i}\left(\overline{\mathcal U_{jn}}\{(\mathcal A_0)_{jk},\mathcal U_{kn}\}
+(\mathcal A_0)_{jk}\{\overline{\mathcal U_{jn}},\mathcal U_{kn}\}
+\mathcal U_{kn}\{\overline{\mathcal U_{jn}},(\mathcal A_0)_{jk}\} \right)
\]

Interverting $j$ and $k$ in half of the terms and noticing that, since $\mathcal A_0$ is Hermitian, 
$(\mathcal A_0)_{jk}=\overline{(\mathcal A_0)_{kj}}$, we get easily \eqref{theformula}.

\subsection{Symbolic properties}\label{symbol}$ $

With the hypothesis that both $\mathcal D$ and $\mathcal U$ are pseudodifferential matrices it is quite obvious that $V_\e$
and $D_\e$ are pseudodifferential matrices as well. Indeed the formal construction in the preceding section shows that the
iterative process uses only three things: multiplications of pseudodifferential operators, computation of subprincipal
symbols and solving equation \eqref{secular}. 
%

For \eqref{secular}, the formula \eqref{normal-form} used in the proof of Lemma \ref{pert}, together with the non-degeneracy
condition \eqref{nondeg} which shows clearly that $(\delta_i(x,\xi)-\delta_j(x,\xi))^{-1}$ is a symbol, implies that $W_0$ is
a pseudodifferential operator. 

 Note that the microlocalization assumption is crucial in order that the expansions obtained by this iterative construction
do  define symbols. We have indeed no uniform control on the growth at infinity.

\subsection{The Rossby-Poincar\'e case}\label{rosscase}$ $

In the case of oceanic waves $A(x,\e D,\e)$ is given by \eqref{defA}:

$$
A (x,\eps D):=   i\left( \begin{matrix}   \eps\bar u\cdot \eps \nabla  &\eps \d_1&\eps \d_2 \\
\eps \d_1 & \eps\bar u\cdot \eps \nabla +\eps^2 \partial_1 \bar u_1  & -b( x_2)+\eps^2 \partial_2 \bar u_1 \\
\eps \d_2 & b( x_2) +\eps^2 \partial_1 \bar u_2&  \eps\bar u\cdot \eps \nabla +\eps^2 \partial_2 \bar u_2 \end{matrix}
\right) \,.
$$
Therefore 
$$
\mathcal A_0 (x,\xi):=   \left( \begin{matrix}   
0  &\xi_1&\xi_2 \\
\xi_1 & 0& -ib( x_2) \\
\xi_2 &i b( x_2) & 0  \end{matrix} \right) \,,
$$
and 
$$
\mathcal A_1 (x,\xi):=   \left( \begin{matrix}   
\bar u\cdot \xi  &0&0 \\
0 & \bar u\cdot \xi   &0 \\
0&0& \bar u\cdot \xi \end{matrix} \right) \,=\bar u\cdot \xi\left( \begin{matrix}   
1  &0&0 \\
0 & 1  &0 \\
0&0&1 \end{matrix} \right).
$$
A straightforward computation shows that the spectrum of $\mathcal A_0$ is 
$$\left\{0,\sqrt{\xi_1^2+\xi_2^2+b^2(x_2)},-\sqrt{\xi_1^2+\xi_2^2+b^2(x_2)}\right\}.$$

\subsubsection{Microlocalization}$ $
 The three eigenvalues  of $\mathcal A_0$ are separated if and only if  $$\xi_1^2+\xi_2^2+b^2(x_2)\neq 0.$$
 Therefore, considering  a compact subset ${\mathcal K} $ of $ \R^{2d}$ such that
 $${\mathcal K} \cap \{(x_1,x_2,\xi_1,\xi_2)\,/\, \xi_1^2+\xi_2^2+b^2(x_2)= 0\} =\emptyset\,$$
 ensures that
 \begin{itemize}
 \item the eigenvalues do not cross,  so that it is possible to get a
unitary diagonalizing matrix with  regular entries;
 \item the non degeneracy condition  \eqref{nondeg} is satisfied.
 \end{itemize}
 
 In other words, $A(x,\e D,\e)$ satisfies the assumptions of Theorem \ref{dia} provided that one considers only its action 
 on vector fields which are suitably microlocalized.
 
 We assume of course that this microlocalization condition is satisfied by the initial datum, which is the condition (\ref{cond1}).
 
 Furthermore, we will prove in the next two sections that the propagation by the scalar operators $T_\pm$ and $T_R$ (to be defined now) preserves this suitable microlocalization,
 thus justifying a posteriori  the diagonalization procedure for all times.

 \subsubsection{Computation of the Poincar\'e and Rossby Hamiltonians}$ $
The above computations show that one can define the two  Poincar\'e Hamiltonians as follows:
$$
\tau_\pm:= \pm \sqrt{\xi_1^2+\xi_2^2+b^2(x_2)}
$$
and we shall denote the associate operator constructed via Theorem~\ref{dia} by~$T_\pm$.

Now let us consider the Rossby Hamiltonian. 
In all this paragraph, for the sake of readability, we will denote 
$$\xib =\sqrt{\xi_1^2+\xi_2^2+b^2(x_2)}\,.$$
An easy computation shows that a (normalized) eigenvector of $\mathcal A_0(x,\xi)$ of zero eigenvalue is
\[
u_0=\frac1\xib \left(\begin{array}{c}b \\i\xi_2\\-i\xi_1\end{array}\right)
\]

By Theorem \ref{dia}, the Rossby Hamiltonian is then given by  the formula
\be\label{theformulaR}
\tau_R=
\sum_{j,k=1\dots 3}\left(\Im\left(\overline{u_{j0}}
\{a_{jk},u_{k0}\}\right)+\frac{a_{jk}\{\overline{u_{j0}},u_{k0}\}}{2i}\right)
+\sum_{j,k=1\dots 3} (\mathcal A_1)_{jk}\overline{u_{j0}}u_{k0}.
\ee

In order to compute the different Lie brackets, we start with a couple of simple remarks~:
$$ \{ \xi_j,f\} =\d_{x_j} f \quad  \hbox{ and } \quad \{b(x_2), f\} =-b'(x_2) \d_{\xi_2} f\,.$$
In particular, if $f$ does not depend on $x_1$, then $ \{ \xi_1,f\} =0$.

\medskip
The contribution  of the first term in  the
parenthesis in \eqref{theformulaR} is 
$$\begin{aligned}
\sum_{j,k=1\dots 3}&\left(\overline{u_{j0}}
\{a_{jk},u_{k0}\}\right)\\
& = {b\over \xib} \left\{ \xi_2, {-i\xi_1\over \xib}\right\} +{i\xi_2\over \xib} \left\{ ib, {-i\xi_1\over \xib}\right\}  +{i\xi_1\over \xib} \left(\left \{\xi_2, {b\over \xib}\right\} + \left\{ib, {i\xi_2\over \xib}\right\}\right)\\
&=  {-ib\xi_1 \over \xib}\d_{x_2} {1\over \xib} -{i\xi_2\xi_1b' \over \xib} \d_{\xi_2}{1\over \xib} +{i\xi_1\over \xib}\d_{x_2}{b\over \xib}+{i\xi_1b'\over \xib}\d_{\xi_2}{\xi_2\over \xib}\\
&= {2i\xi_1b'\over \xib^2} \cdotp
\end{aligned}
$$
Using the distributivity of the Poisson brackets, we get the contribution  of the second term in a very similar way
$$\begin{aligned}
\sum_{j,k=1\dots 3}&\frac{a_{jk}\{\overline{u_{j0}},u_{k0}\}}{2}\\
&= \xi_1 \left\{ {b\over \xib}, {i\xi_2\over \xib}\right\} -\xi_2 \left\{ {b\over \xib}, {i\xi_1\over \xib}\right\}+ib \left\{ {i\xi_1\over \xib},{i\xi_2\over \xib}\right\} \\
&=  \xi_1 \left( {ib\over \xib} \left\{{1\over \xib}, \xi_2\right\} +{i\xi_2\over \xib} \left\{b, {1\over \xi_b}\right\}+{i\over \xib^2} \{b,\xi_2\} \right) -{i\xi_2\xi_1\over \xib} \left\{ b, {1\over \xib}\right\}-{ib \xi_1\over \xib} \left\{{1\over \xib},\xi_2\right\}\\
&= {-ib'\xi_1\over \xib^2} \cdotp
\end{aligned}
$$ 
The computation of the second term of the right hand side of \eqref{theformulaR} is trivial since $\mathcal A_1$ is  a multiple of the
identity.   Adding with the two previous expressions we get finally 
 $$\tau_R=\frac{\xi_1b'}{\xi_1^2+\xi_2^2+b(x_2)^2} +\bar u \cdot \xi  $$
 and the associate operator will be denoted by~$T_R$.

 \begin{Rem}\label{x1-dependence}
 Since the elementary steps of the diagonalization process use only multiplications, computations of subprincipal symbols  and solving normal forms equations, all the subsymbols of $T_R$ and $T_\pm$ depend on $x_1$ only through  $\bar u$ and its derivatives.
 \end{Rem}


\section{Study of the Poincar\'e waves}\label{sectionpoincare} 
In   Section~\ref{scalar} we constructed two linear operators, called~$T_\pm$, whose principal symbols are
$$
\tau_\pm= \pm \sqrt{\xi_1^2+\xi_2^2+b^2(x_2)}.
$$
We now want to study the propagation equation associated to those operators, namely  the linear equation in~$\R \times \R^2$
\begin{equation}\label{eqpoincare}
i \e^2\partial_t \varphi_\pm = T_\pm \varphi_\pm  , \quad \varphi_{\pm |t = 0} = \varphi^0_\pm
\end{equation}
where~$\varphi^0_\pm$ are microlocalized in a compact set~${\mathcal C}$ satisfying Assumption~(\ref{cond1}). Before studying that equation we need to check that it makes sense, since a priori~$T_\pm$ is only defined on vector fields microlocalized on such a compact set. This is achieved in the coming section, where we check that  the separation of eigenvalues~(\ref{nondeg}) required in the statement of 
 Theorem~\ref{dia}  holds because~$\displaystyle \sqrt{\xi_1^2+\xi_2^2+b^2(x_2)}$ remains bounded away from zero during the propagation. 

Then we shall show that   the solutions to these equations exit any compact set in finite time (Point~(\ref{poincarebehaviour}) of Theorem~\ref{mainresult}).

\subsection{Microlocalization}
Let us prove the following result, which provides the first part of Point~(\ref{nondegeneracy}) in Theorem~\ref{mainresult} and allows to make sense of Equation~(\ref{eqpoincare}) for all times.
  \begin{Prop} Under the assumptions of Theorem~\ref{mainresult}, the operators~$T_\pm$ are self-adjoint, and 
    the function~$\varphi(t) = e^{ i \frac t{\e^2} T_\pm} \varphi^0_\pm $  are such that~$\mu \mbox{Supp}  \varphi _\pm(t) $ satisfies~(\ref{cond1}) for all times.
  \end{Prop}
 \begin{proof}
 The proof of that result relies on a spectral argument. Due to the form of the principal symbols of~$T_\pm$ recalled above, the operators~$T_\pm$ are
  self-adjoint. We can therefore define two  families~$(\psi_n^\pm)_{n \in \N}$ of (pseudo)-eigenvectors of~$T_\pm$ in~$L^2(\R^2)$ and two sequences of eigenvalues~$\lambda_n^\pm$ such that if the initial data writes
 $$
  \varphi^0_\pm (x) = \sum_n c_n^{\pm,0} \psi_n^\pm(x),
 $$
 then
 $$
  \varphi_\pm (t,x) = \sum_n e^{ i \frac {\lambda_n^\pm t}{\e^2}} c_n^{\pm,0} \psi_n^\pm(x).
 $$
 Since the eigenfunctions~$\psi_n^\pm$ are microlocalized on the energy surfaces of the Poincar\'e Hamiltonians, the result follows.
 \end{proof}
\subsection{Dispersion}
In this paragraph we shall prove  Point~(\ref{poincarebehaviour}) of Theorem~\ref{mainresult}.
The strategy is the following.

 In Section~\ref{shorttime} we prove using semi-classical analysis that for a very short time, the solutions to~(\ref{eqpoincare}) remain microlocalized in a compact set satisfying assumption (\ref{cond2}), and such that $ \mu \mbox{Supp}_{x_1} \varphi_\pm $ become disjoint from $\supp_{x_1} \bar u$. Section~\ref{longtime} is then devoted to the long-time behaviour of the solution, and Mourre estimates allow to prove that the solution exits any compact set after some time, and that it remains microlocalized far from~$\xi_1 = 0$.  

%

%


\subsubsection{Short time behaviour}\label{shorttime}
The aim of this paragraph is to prove the following result. It shows that the solutions of~(\ref{eqpoincare}) exit the support of~$\bar u$ after a time~$t_{\rm exit}\eps$, for~$| t_{\rm exit} | $ large enough (independent of~$\e$).
We only state
the forward in time result:  the backwards result is identical, up to changing the sign of time.
We shall further restrict  the analysis to~$T_+$ since   the argument for~$T_-$ is identical, up to some sign changes. 

 \begin{Prop}\label{shorttimeprop}
Let~$\varphi^0$ be a function, microlocalized in a compact set~${\mathcal C}$ satisfying Assumption~(\ref{cond1}), and let~$\varphi$ be the associate solution of~(\ref{eqpoincare}). Let~$[u_-,u_+]$ be a closed interval of~$\R $ containing~$\mbox{Supp}_{x_1}  \bar u$. There exists a  constant~$t_{\rm {exit}}>0$ such that for any~$\e \in ]0,1[$, the function~$\varphi (\e t_{\rm {exit}}, \cdot)$ is microlocalized in a compact set~${\mathcal K} $ such that the projection of~${\mathcal K} $ onto the~$x_1$-axis does not interesect~$[u_-,u_+]$. Moreover~$ \mu \mbox{Supp}_{\xi_1} \varphi$ is unchanged.

More precisely, if~$\mu \mbox{Supp}_{\xi_1} \varphi_0 \subset \R^+ \setminus \{0\}$, then~$ \mu \mbox{Supp}_{x_1} \varphi(\e  t_{\rm {exit}} , \cdot)\subset ]u_+,+\infty[$, and if~$\mu \mbox{Supp}_{\xi_1} \varphi_0 \subset \R^- \setminus \{0\}$, then~$ \mu \mbox{Supp}_{x_1} \varphi(\e  t_{\rm {exit}}, \cdot)\subset ]-\infty, u_-[$.
\end{Prop}
\begin{proof}
Define the function~$\psi (s) := \varphi ( \e s)$. Then~(\ref{eqpoincare}) reads
\begin{equation}\label{eqpoincarebis}
i \e \partial_s \psi= T_+ \psi , \quad \psi_{|s = 0} = \varphi^0,
\end{equation}
and any result proved on~$\psi$ on~$[0,{\mathcal T}]$ will yield the same result for~$\varphi$ on~$[0,{\mathcal T}\eps]$. Notice that~(\ref{eqpoincarebis}) is written in a semi-classical setting, so by the propagation of the microsupport theorem (see for instance~\cite{MA} and the references therein), the microsupport of~$\psi $ is propagated by the bicharacteristics, which are the integral curves of the principal symbol. Recall that the principal symbol of~$T_+$ is
$$
 \tau_+ (\xi_1,x_2 ,\xi_2)=   \sqrt{\xi_1^2 + \xi_2^2 + b^2(x_2)} 
$$
and the bicharacteristics are given by the following set of ODEs:
$$ 
  \left\{ \begin{array}{llll}
\displaystyle \dot x^t = \displaystyle  {\nabla_{\xi }  \tau_+ }  (\xi_1^t,x_2^t,\xi_2^t) , & \qquad x^0   =  (x_1^0,x_2^0)   \\
\displaystyle\dot \xi^t= - \displaystyle   {\nabla_x  \tau_ + }  
  (\xi_1^t,x_2^t,\xi_2^t) ,& \qquad \xi^0 =  (\xi_1^0,\xi_2^0) . 
\end{array} \right. $$
Notice that~$ \tau_+ $ is independent of~$x_1$, so~$\dot \xi_1^t$ is identically zero and therefore~$\xi_1^t \equiv \xi_1^0$. So for all~$s \geq 0$, the microlocal support  in~$ \xi_1$ of~$\psi (s)$ remains unchanged, and in particular is far from~$\xi_1 = 0$. Moreover one has
$$
\dot x_1^t =  \frac {\xi_1^0 }{ \sqrt{(\xi_1^0 )^2 + (\xi_2^t)^2 + b^2(x_2^t)} } \cdotp
$$
Now we recall that the bicharacteristic curves lie on  energy surfaces, meaning that on each bicharacteristic, $\tau_+ (\xi_1^0,x_2^t,\xi_2^t) $ is a constant. That implies that  $ (\xi_2^t)^2 + b^2(x_2^t)$ is a constant on each bicharacteristic, so that for all times,
$$
\dot x_1^t \equiv  \frac {\xi_1^0 }{ \sqrt{(\xi_1^0 )^2 + (\xi_2^ 0)^2 + b^2(x_2^ 0)} }   \cdot
$$
If~$\xi_1^0 > 0$,  
then~$x_1$ is propagated to the right and eventually escapes to the right of the support in~$x_1$ of~$\bar u$, whereas if~$\xi_1^0 < 0$,  the converse (to the left) occurs. 
Proposition \ref{shorttimeprop} is proved.
\end{proof}
%
%
%
\subsubsection{Long time behaviour}\label{longtime}
The aim of this paragraph is to prove the following result, which again is only proved for positive times for simplicity.
 \begin{Prop}\label{longtimeprop}
Under the assumptions of Proposition~\ref{shorttimeprop}, let~$\varphi^+$ be the solution of~(\ref{eqpoincare})
associated with the data~$\varphi (\e  t_{\rm {exit}}, \cdot)$. Then~$\mu\mbox{Supp}_{x_1}\varphi^+(t)$  does not intersect~$\mbox{Supp}_{x_1}\bar u$ for~$t \geq \e  t_{\rm {exit}}$, and $\mu\mbox{Supp}_{\xi_1} \varphi^+ (t)$ remains unchanged for~$t \geq \e  t_{\rm {exit}}$. Finally~$\mu\mbox{Supp}_{x_1}\varphi^+(t)$ exits  any compact set in~$x_1$ in finite time. \end{Prop}

\begin{proof} Before going into the proof, we shall simplify the analysis by only studying the case of~$T_+$ (the case~$T_-$ is obtained by identical arguments), and we shall only deal with the case when the support in~$\xi_1$ of the data lies in the positive half space.  The other case  is obtained   similarly.

The proof is based on Mourre's theory which we shall  now briefly recall, and  we refer to~\cite{JMP} and~\cite{Hunz} for all details.
Let us consider two self-adjoint operators~$H$ and~$A$ on a Hilbert space~${\mathcal H}$. We make the following assumptions:

\begin{enumerate}
\item \label{A1} the intersection of the domains of~$A $ and~$H$ is dense in the domain~${\mathcal D} (H)$ of~$H$  .

\item  \label{A2}  $t \mapsto e^{itA} $ maps~${\mathcal D} (H)$ to itself, and for all~$ \varphi^0 \in {\mathcal D} (H)$,
$$
\sup_{t \in [0,1]} \|H e^{itA} \varphi^0\| < \infty .
$$

\item  \label{A3} The operator $ i[H,A]$ is bounded from below and closable, and the domain~${\mathcal D} (B_1)$  where~$i B_1$ is its closure, contains~${\mathcal D} (H)$. More generally for all~$n \in \N$ the operator~$ i [i B_{n},A] $  is bounded from below and closable  and the domain~${\mathcal D} (B_{n+1})$   of its closure~$i B_{n+1}$ contains~${\mathcal D} (H)$, and finally~$B_{n+1} $ extends to a bounded operator from~${\mathcal D} (H)$ to its dual. 

\item There exists~$\theta >0$  and an open interval~$\Delta$ of~$\R$ such that if~$E_\Delta$ is the corresponding spectral projection of~$H$, then
\begin{equation} \label{mourreestimate}
E_\Delta i  [H,A] E_\Delta \geq \theta E_\Delta .
\end{equation}

\end{enumerate} 
Note that Assumptions~(\ref{A1} - \ref{A3}) can be replaced by the fact that~$[f(H),A]$ is bounded for any smooth, compactly supported function~$f$ (see~\cite{Hunz}).

Under those assumptions, for any integer~$m \in \N$ and  for any~$\theta ' \in ]0 , \theta [$, there is a constant~$C$ such that
 $$
\|  \chi_- (A-a-\theta't) e^{-iHt} g(H)  \chi_+ (A-a ) \| \leq C  t^{-m}
$$
where~$\chi_\pm$ is the characteristic function of~$\R^\pm$, $g $ is any smooth compactly supported function in~$\Delta$, and the above bound is uniform in~$a \in \R$. As pointed out in~\cite{Hunz}, this implies in particular that for any~$\varphi^0$ in the image of~$E_\Delta$, the function~$e^{-iHt} \varphi^0$ has spectral support in~$[a + t \theta' , \infty[$ with respect to~$A$, up to~$t^{- \infty}$.

Let us apply this theory to our situation. We   consider equation~(\ref{eqpoincare}) with   data~$e^{i \frac {  t_{\rm {exit}} } {\e} T_+}  \varphi^0_{+ }$,  and let us define the operator~$T_+^0$ as the operator~$T_+$ where~$\bar u$ has been chosen identically zero. We shall start by studying the equation 
\begin{equation}\label{eqpoincarebaruzero}
i \e^2\partial_t \widetilde \varphi = T_+^0 \widetilde\varphi  , \quad \widetilde\varphi_{|t =\e  t_{\rm {exit}} } =e^{i \frac {  t_{\rm {exit}} } {\e} T_+}  \varphi^0_{+ },
\end{equation}
for which we shall prove Proposition~\ref{longtimeprop}. Then we shall prove that the solution~$\widetilde\varphi$ actually solves the original equation~(\ref{eqpoincare}) with the same data~$e^{i \frac {  t_{\rm {exit}} } {\e} T_+}  \varphi^0_{+ }$ at $t=\eps t_{\rm exit}$ up to~$O(\e^\infty)$, because its support in~$x_1$ lies outside the support of~$\bar u$ and because the symbolic expansion of~$T_+$  depends on $x_1$ only through $\bar u$ and its derivatives (see Remark \ref{x1-dependence}).

So let us start by applying Mourre's theory to~(\ref{eqpoincarebaruzero}). 
Let us write  the projection of~${\mathcal K}$ onto the~$\xi_1$-axis as included in~$[d_0,d_1]$ with~$0<d_0<d_1< \infty$. We recall that   on the support of~$e^{i \frac {  t_{\rm {exit}} } {\e} T_+}  \varphi^0_{+ }$, $x_1 $ remains to the right of the support of~$\bar u$. Then we apply the theory to~$H =  T_+^0 $ and to~$A = x_1$ (the pointwise multiplication). 
Assumptions~(\ref{A1}) to~(\ref{A3}) are easy to check, in particular because this is a semiclassical setting, so only the principal symbols need to be computed.  Similarly finding a lower bound for~$E_\Delta i  [  T_+^0,x_1] E_\Delta$ boils down to computing the Poisson bracket~$ \{ \tau_+ , x_1 \}  $ where
$$
\{ f , g \}   = \nabla_\xi  f \cdot \nabla_x g -  \nabla_x f \cdot \nabla_\xi g,
$$ 
 and one finds
\begin{equation}\label{lowerbound}
  \{ \tau_+ , x_1 \}  =  \frac{\xi_1}{  \sqrt{\xi_1^2 + \xi_2^2 + b^2(x_2)}  } \cdotp
\end{equation}
Since~$  T_+^0$ has constant coefficients in~$x_1$,  one can take the Fourier transform of~(\ref{eqpoincarebaruzero})  and~$\xi_1 $ is preserved,  so in particular for all times one has~$\mu\mbox{Supp}_{\xi_1} \widetilde \varphi(t) \subset [d_0,d_1]$. One can furthermore choose for~$\Delta$ an interval of~$\R$ of the type~$]D_0  ,  D_1[$ where the constants~$D_0$ and~$D_1$ are chosen so that for any~$(x,\xi) \in {\mathcal K}$, one has
\begin{equation}\label{D0D1}
D_0< \sqrt{\xi_1^2 + \xi_2^2 + b^2(x_2)} < D_1.
\end{equation}
As the microlocal supports of the
eigenfunctions of~$T_+^0$ lie  on energy surfaces, we know that the solution  to~(\ref{eqpoincarebaruzero})   will remain in~$E_\Delta$ for all times. 
Now let us apply the results of~\cite{JMP} and~\cite{Hunz}.  By Lemma \ref{sub}, (\ref{lowerbound}), (\ref{D0D1}) and the assumption on~$\xi_1$ written above, we have that
$$
E_\Delta i  [H,A] E_\Delta \geq \e \frac{d_0}{D_0} E_\Delta ,
$$
so~(\ref{mourreestimate}) holds with~$\theta = \e d_0 /D_0$. It follows that the solution~$\displaystyle e^{i \frac {(t-\eps t_{\rm exit})}{\e^2}  T_+^0} \bigl(e^{i \frac {  t_{\rm {exit}} } {\e} T_+}  \varphi^0_{+ }\bigr)$ to~(\ref{eqpoincarebaruzero}) has a support in~$x_1$ such that 
$$
x_1 >  u_+ + \frac{d_0}{D_0} \frac t \e   
$$
which proves the result for~(\ref{eqpoincarebaruzero}).

Since~$\mu\mbox{Supp}_{x_1} e^{i \frac {(t-\eps t_{\rm exit})}{\e^2}  T_+^0}\bigl(e^{i \frac {  t_{\rm {exit}} } {\e} T_+}  \varphi^0_{+ }\bigr)$ does not cross~$\mbox{Supp}_{x_1}\bar u$, one has actually
 $$
 e^{i \frac {(t-\eps t_{\rm exit})}{\e^2}  T_+^0}\bigl(e^{i \frac {  t_{\rm {exit}}}  {\e} T_+}  \varphi^0_{+ }\bigr)=  e^{i \frac {(t-\eps t_{\rm exit})}{\e^2}  T_+ }\bigl(e^{i \frac {  t_{\rm {exit}}}   {\e} T_+}  \varphi^0_{+ }\bigr) \quad \mbox{in} \: L^2
 $$ 
locally  uniformly in~$t$ (see Appendix).

The proposition follows.
\end{proof}

 \section{Propagation of the Rossby waves}\label{rossby}

  \subsection{Semiclassical transport equations and microlocalization}$ $
  
  Because of the scaling of the Rossby hamiltonian (which is smaller  than the Poincar\'e hamiltonians by one order of magnitude), on the times scales considered here, the propagation of energy by Rossby waves 
  is described by the hamiltonian dynamics
  $$
  {dx_i\over dt} ={\d \tau_R\over \d \xi_i},\quad {d\xi_i\over dt} = - {\d \tau_R\over \d x_i},$$
  which can be written explicitly
  \begin{equation}
  \label{Rossby-eq}
  \begin{aligned}
  {d x_1\over dt} &=b'(x_2){\xib^2-2\xi_1^2\over \xib^4 }+u_1(x) ,\\
  {dx_2\over dt} &=-2 b'(x_2) {\xi_1\xi_2 \over \xib^4} +u_2(x),\\
  {d\xi_1\over dt} &=-\d_1 u_1 (x) \xi_1-\d_1 u_2(x) \xi_2,\\
  {d\xi_2\over dt} &=\xi_1{  2b(b')^2-b''\xib \over \xib^4} - \d_2 u_1 (x) \xi_1-\d_2 u_2(x) \xi_2
  \end{aligned}
  \end{equation}
  where we recall that~$\xib =\sqrt{\xi_1^2+\xi_2^2+b^2(x_2)}\,.$
  In order for the dynamics to be well defined and also in order to justify the diagonalization process, we   need   the quantity $\xib$ to remain bounded from below for all times.
  
  \begin{Prop}
  Let $\CC$ be some compact subset of $\R^4$ such that
   $${\mathcal C} \cap \{(x_1,x_2,\xi_1,\xi_2)\,/\, \xi_1^2+\xi_2^2+b^2(x_2)= 0\} =\emptyset\,.$$
      Then the bicharacteristics $t\mapsto (x(t),\xi(t))$ of the Rossby Hamiltonian starting from any point $ (x_1^0,x_2^0,\xi_1^0,\xi_2^0)$  of $\CC$ are defined  globally in time,
   and $\forall t \in \R$,
   $$\inf_{(x_1^0,x_2^0,\xi_1^0,\xi_2^0)\in \CC} (   \xi_1(t)^2+\xi_2(t)^2+b^2(x_2(t)) >0.$$
  \end{Prop}
  
  \begin{proof}
  As $b'$, $b''$, $u$ and $Du$ are Lipschitz,
  by the Cauchy-Lipschitz theorem the system of ODEs (\ref{Rossby-eq}) has a unique maximal solution. In order to prove that this solution is defined globally, it is enough to prove that the time derivative of this solution is uniformly bounded. This comes from assumption (\ref{bsymbol}) giving an upper bound on $b'/\bigl(1+b^2(x_2) \bigr)^\frac12$ and $b''/\bigl(1+b^2(x_2) \bigr)^\frac12$, and from the
  lower bound on $\xib$ to be established now.
  
  The crucial assumption here is the fact that $b'$ and $b$ do not vanish simultaneously, and more precisely the existence of $\eta, \beta >0$ such that for all $x_2\in \R$.
  $$ |b(x_2)|< \eta \Rightarrow |b'(x_2)|\geq \beta\,.$$
  Along a trajectory of the Rossby Hamiltonian, $\tau_R$ is conserved. Therefore, 
  \begin{itemize}
  \item either $b(x_2) \geq\eta$, and $\xib^2 >\eta^2$,
  \item or $b(x_2) <\eta$. In this last case, if $|\xi|\geq \eta$,  $\xib^2 >\eta^2$. Else,
  $$|\tau| \geq {\eta \over \xib}- \|\bar u\|_\infty \eta $$
  from which we deduce that
  $$ \xib \geq {\eta \over |\tau |+\| \bar u\|_\infty \eta}\,\cdotp$$
  \end{itemize}
  In any case,
  we have the lower bound
  $$ \xib \geq \eta \min\left( 1, {1 \over |\tau |+\|\bar  u\|_\infty \eta}\right)\,,$$
  which can be made uniform  for initial data in $\CC$ replacing $|\tau |$ by $\max_\CC |\tau|$.
    \end{proof}
  
  \subsection{Dynamics outside from the support of $\bar u$}$ $
  
  Using the fact that $\bar u$ has compact support, and simple properties of the Rossby dynamics in the absence of zonal flow, we can prove the following
  
   \begin{Prop} 
   Let $\CC$ be some compact subset of $\R^4$ such that
   $${\mathcal C} \cap \{(x_1,x_2,\xi_1,\xi_2)\,/\, \xi_1^2+\xi_2^2+b^2(x_2)= 0\}=\emptyset\,.$$
      Then 
   the bicharacteristics of the Rossby Hamiltonian starting from any point  of $\CC$ are bounded in $x_2$:
    $$\sup_{(x_1^0,x_2^0,\xi_1^0,\xi_2^0)\in \CC}|x_2(t)| <\infty.$$
    \end{Prop}

 \begin{proof}
 $\bullet$ Let us start by describing the dynamics in the absence of zonal flow: $\xi_1$ is then an invariant of the motion,
 so that the dynamics in $(x_2,\xi_2)$ can be decoupled. Furthermore, as the energy surfaces are compact
 $$\tau = {\xi_1 b'(x_2) \over \xi_1^2+\xi_2^2+b^2(x_2)}$$
 the motion along $x_2$ is  periodic (with infinite period for homoclinic and heteroclinic orbits).
 
 The motion along $x_1$ is then determined by the equation
 $${d x_1\over dt} =b'(x_2){\xib^2-2\xi_1^2\over \xib^4 } \cdotp $$
 It is trapped if and only if the average of the right-hand side over one period is zero.
 Outside from saddle points, this quantity depends continuously on $\xi_1$, so that we expect the initial data leading to trapped trajectories to belong to a manifold of codimension 1.
 This can be proved rigorously if $b^2$ has only one non degenerate critical points (see \cite{CGPS}).

 \medskip
 $\bullet$ Let us now turn to the influence of the zonal flow. We will first check that the only possible escape direction is again $x_1$.
 Indeed the energy surfaces corresponding to  $\tau \neq 0$ are bounded in the $x_2$ direction~: as $x_2\to \pm \infty$,
 $${b'(x_2) \xi_1 \over \xib^2} +\bar u(x)\cdot \xi \to 0\,.$$
 Consider now  a trajectory on the energy level $\tau_R=0$, and some point of this trajectory $(y_1,y_2,\xi_1,\xi_2) $ such that  $y_2 \notin \supp_{x_2}\bar u$. One has
 $$b'(y_2) \xi_1=0\,.$$
 - If $b'(y_2) =0$, then 
 $$ {d x_1\over dt} = {dx_2\over dt} ={d\xi_1\over dt} = {d\xi_2\over dt} =0\,.$$
 The uniqueness in Cauchy-Lipschitz theorem implies then that the trajectory is nothing else than a fixed point, and therefore in particular is bounded.
 
- If $\xi_1 =0$, then
 $$  {dx_2\over dt} ={d\xi_1\over dt} = {d\xi_2\over dt} =0 \quad \hbox{ and } \quad {dx_1\over dt} = b'(y_2){\xi_2^2+b^2(y_2)-\xi_1^2\over \xib^4 }\,,$$
 meaning that the trajectory is a uniform straight motion along $x_1$. In particular, it is bounded in the $x_2$-direction.
 
 Finally, we conclude that trajectories on the energy level $\tau_R=0$ are either trapped in the support $\supp_{x_2}\bar u$, or trivial in the $x_2$-direction.
 \end{proof}

  
  \subsection{Perspectives}$ $
As recalled in the introduction, it is generally   believed that  in the situation depicted in this paper (a flow around a   large macroscopic current) Rossby waves are trapped. However due to the   2-dimensional setting (compared to the work in~\cite{CGPS}) the trapping in the~$x_1$ direction seems difficult to prove, outside some specific cases studied in the previous paragraph. One way to be convinced of the trapping phenomenon should be by implementing the dynamical system  numerically. It should be pointed out however that actually 
 in order to get physically relevant predictions for the oceanic eddies, one should consider 3D models, or at least 2D models involving the influence of stratification.
The methods presented here seem to be robust and should be extended to such complex models, up to again the study of the hamiltonian system describing the Rossby dynamics.

\section*{ Appendix~: a comparison result}

For the sake of completeness, we state here the result which shows the stability of the propagation under a $O(\eps^\infty)$ error on the propagator.
This result has been implicitly used in the proof of the diagonalization when comparing $A$ and $T_\pm$, $T_R$, and in the proof of dispersion when comparing $T_\pm$ and $T^0_\pm$.

\begin{Prop}
Let $A_\eps$ and $\tilde A_\eps$ be two pseudo-differential operators such that
\begin{itemize}
\item $iA_\eps$ is  hermitian in $L^2(\R^d)$,
\item $A_\eps-\tilde A_\eps =O(\eps^\infty)$ microlocally on $\Omega \subset \R^{2d}$.
\end{itemize}
Let $\tilde \varphi$ be a solution to
$$i\d_t\tilde  \varphi +\tilde A_\eps \tilde \varphi =0$$
microlocalized in $\Omega$, and $\varphi$ be the solution to
$$i\d_t  \varphi + A_\eps  \varphi =0$$
with the same initial data.
Then, for all $N\in \N$,
$$\sup_{ t\leq \eps^{-N}} \| \varphi(t)-\tilde \varphi(t)\|_{L^2(\R^d)} =O(\eps^\infty)\,.$$
\end{Prop}

\begin{proof}
The proof is based on a simple energy inequality and is completely straightforward. We have
$$\begin{aligned}
{d\over dt} \| \varphi-\tilde \varphi\|^2_{L^2(\R^d)} &=2\langle iA_\eps\varphi-i\tilde A_\eps \tilde \varphi | \varphi-\tilde \varphi \rangle\\
&= 2\langle (iA_\eps-i\tilde A_\eps) \tilde \varphi | \varphi-\tilde \varphi \rangle\\
&\leq 2\| (A_\eps-\tilde A_\eps) \tilde \varphi \|_{L^2(\R^d)} \| \varphi-\tilde \varphi\|_{L^2(\R^d)}\,.
\end{aligned}
$$
This leads to
$$\| \varphi(t)-\tilde \varphi(t)\|^2_{L^2(\R^d)}=O(\eps^\infty) t\,,$$
which concludes the proof.
\end{proof}

   {\bf Acknowledgements.}  $ $ The authors are grateful to J.-F. Bony and N. Burq    for introducing them to     Mourre estimates, and for interesting discussions. 
I. Gallagher and L. Saint-Raymond are 
partially supported by the French Ministry of Research grant
ANR-08-BLAN-0301-01.

\end{document}